\documentclass[12pt,a4paper]{article}
\usepackage{latexsym,amssymb,amsmath,amsthm,enumerate}
\usepackage{epsfig,graphicx,bm}
\usepackage[english]{babel}
\usepackage[latin2]{inputenc}
\usepackage{color}
\usepackage{t1enc}

\newtheorem{thm}{Theorem}

\newtheorem{lemma}[thm]{Lemma}

\newtheorem{conjecture}[thm]{Conjecture}
\newtheorem{proposition}[thm]{Proposition}

\newtheorem{cor}[thm]{Corollary}

    \newtheoremstyle{TheoremNum}
        {\topsep}{\topsep}              
        {\itshape}                      
        {}                              
        {\bfseries}                     
        {.}                             
        { }                             
        {\thmname{#1}\thmnote{ \bfseries #3}}
    \theoremstyle{TheoremNum}
    \newtheorem{thmn}{Theorem}

\newcommand{\eps}{\varepsilon}
\newcommand{\cP}{\mathcal P}
\newcommand{\cA}{\mathcal A}
\newcommand{\cC}{\mathcal C}
\newcommand{\cD}{\mathcal D}
\newcommand{\cE}{\mathcal E}
\newcommand{\cF}{\mathcal F}
\newcommand{\cH}{\mathcal H}
\newcommand{\cG}{\mathcal G}
\newcommand{\cM}{\mathcal M}

\newcommand{\cS}{\mathcal S}
\newcommand{\cU}{\mathcal U}
\newcommand{\bC}{\mathbf C}

\newcommand{\floor}[1]{\left\lfloor{#1}\right\rfloor}
\newcommand{\oa}[1]{\overrightarrow{#1}}

\title{Supersaturation, counting, and randomness in forbidden subposet problems}
\author{D\'aniel Gerbner$^1$ \and D\'aniel T. Nagy$^1$ \and Bal\'azs Patk\'os$^{1,2}$ \and  M\'at\'e Vizer$^1$
}
\date{
\small $^1$ Alfr\'ed R\'enyi Institute of Mathematics, Budapest \\
\small $^2$ Moscow Institute of Physics and Technology\\
}

\begin{document}

\maketitle

\begin{abstract}
In the area of forbidden subposet problems we look for the largest possible size $La(n,P)$ of a family $\mathcal{F}\subseteq 2^{[n]}$ that does not contain a forbidden inclusion pattern described by $P$. The main conjecture of the area states that for any finite poset $P$ there exists an integer $e(P)$ such that $La(n,P)=(e(P)+o(1))\binom{n}{\lfloor n/2\rfloor}$.


In this paper, we formulate three strengthenings of this conjecture and prove them for some specific classes of posets. (The parameters $x(P)$ and $d(P)$ are defined in the paper.)

\begin{itemize}
    \item 
    For any finite connected poset $P$ and $\varepsilon>0$, there exists  $\delta>0$ and an integer $x(P)$ such that for any $n$ large enough, and $\mathcal{F}\subseteq 2^{[n]}$ of size $(e(P)+\varepsilon)\binom{n}{\lfloor n/2\rfloor}$, $\mathcal{F}$ contains at least $\delta n^{x(P)}\binom{n}{\lfloor n/2\rfloor}$ copies of $P$.
    \item
    The number of $P$-free families in $2^{[n]}$ is $2^{(e(P)+o(1))\binom{n}{\lfloor n/2\rfloor}}$.
    \item
    For any finite poset $P$, there exists a positive rational $d(P)$ such that if $p=\omega(n^{-d(P)})$, then the size of the largest $P$-free family in $\mathcal{P}(n,p)$ is $(e(P)+o(1))p\binom{n}{\lfloor n/2\rfloor}$ with high probability.
\end{itemize}

    
    
\end{abstract}

\section{Introduction}

Extremal set theory starts with the seminal result of Sperner \cite{Spe28} that was generalized by Erd\H os \cite{Erd45} as follows: if a family $\cF\subseteq 2^{[n]}$ of sets does not contain a nested sequence $F_1\subsetneq F_2 \subsetneq \dots \subsetneq F_{k+1}$ (such nested sequences are called \textit{chains of length $k+1$} or \textit{$(k+1)$-chains} for short), then its size cannot exceed that of the union of $k$ middle levels of $2^{[n]}$, i.e., $|\cF|\le \sum_{i=1}^k\binom{n}{\floor{\frac{n-k}{2}}+i}$. This theorem has many applications and several of its variants have been investigated. 

In the early 80's, Katona and Tarj\'an  \cite{KatTar83} introduced the following general framework to study set families avoiding some fixed inclusion patterns: we say that a subfamily $\cG$ of $\cF$ is a \textit{(non-induced) copy} of a poset $(P, \le)$ in $\cF$, if there exists a bijection $i:P\rightarrow \cG$ such that if $p,q \in P$ with $p\le q$, then $i(p) \subseteq i(q)$. If $i$ satisfies the property that for $p,q\in P$ we have $p\le q$ if and only if $i(p)\subseteq i(q)$, then $\cG$ is called an \textit{induced copy} of $P$ in $\cF$. If $\cF$ does not contain any (induced) copy of $P$, the $\cF$ is said to be \textit{(induced) $P$-free}. The largest possible size of a(n induced) $P$-free family $\cF\subseteq 2^{[n]}$ is denoted by $La(n,P)$ ($La^*(n,P)$). Let $P_k$ denote the $k$-chain, then the result of Erd\H os mentioned above determines $La(n,P_{k+1})$. These parameters have attracted the attention of many researchers, and there are widely believed conjectures in the area (see Conjecture \ref{regisejtes}) that appeared first in \cite{Buk09} and \cite{GriLu09}, giving the asymptotics of $La(n,P)$ and $La^*(n,P)$.

Let $e(P)$ denote the maximum integer $m$ such that for any $i\le n$, the family $\binom{[n]}{i+1}\cup \binom{[n]}{i+2}\cup \dots \cup \binom{[n]}{i+m}$ is $P$-free. Similarly, let $e^*(P)$ denote the maximum integer $m$ such that for any $i\le n$, the family $\binom{[n]}{i+1}\cup \binom{[n]}{i+2}\cup \dots \cup \binom{[n]}{i+m}$ is induced $P$-free.

\begin{conjecture}\label{regisejtes} \ 

\noindent 
(i)  $La(n,P)=(e(P)+o(1))\binom{n}{\lfloor n/2\rfloor}$.

\noindent 
(ii)  $La^*(n,P)=(e^*(P)+o(1))\binom{n}{\lfloor n/2\rfloor}$.
\end{conjecture}

Conjecture \ref{regisejtes} has been verified for several classes of posets, but is still open in general. For more results on the $La(n,P)$ function, see Chapter 7 of \cite{GerPat18}, and see other chapters for more background on the generalizations considered in this paper.

After determining (the asymptotics of) the extremal size and the structure of the extremal families, one may continue in several directions. \textit{Stability} results state that all $P$-free families having almost extremal size must be very similar in structure to the middle $e(P)$ levels of $2^{[n]}$. \textit{Supersaturation} problems ask for the minimum number of copies of $P$ that a family $\cF\subseteq 2^{[n]}$ of size $La(n,P)+E$ may contain. This is clearly at least $E$, but usually one can say much more. \textit{Counting } problems ask to determine the number of $P$-free families in $2^{[n]}$. As any subfamily of a $P$-free family is $P$-free, therefore the number of $P$-free families is at least $2^{La(n,P)}$. The question is how many more such families there are. Finally, one can address\textit{ random versions} of the forbidden subposet problem. Let $\cP(n,p)$ denote the probability space of all subfamilies of $2^{[n]}$ such that for any $F\subseteq [n]$, the probability that $F$ belongs to $\cP(n,p)$ is $p$, independently of any other set $F'$. What is the size of the largest $P$-free subfamily of $\cP(n,p)$ with high probability\footnote{we say that a sequence of events $E_1,E_2,\dots, E_n,\dots$ holds with high probability (or w.h.p., in short) if $\mathbb{P}(E_n)$ tends to 1 as $n$ tends to infinity}? Clearly, for $p=1$, this is $La(n,P)$. For other values of $p$, an obvious construction is to take a $P$-free subfamily of $2^{[n]}$, and then the sets that are in $\cP(n,p)$ form a $P$-free family. Taking the $e(P)$ middle levels shows that the size of the largest $P$-free family in $\cP(n,P)$ is at least $p(e(P)+o(1))\binom{n}{\floor{n/2}}$ w.h.p.. For what values of $p$ does this formula give the asymptotically correct answer?

\bigskip

In this paper, we will consider supersaturation, counting and random versions of the forbidden subposet problem, mostly focusing on supersaturation results. We will propose three strengthenings of Conjecture \ref{regisejtes} and prove them for some classes of posets. In the remainder of the introduction, we state our results and also what was known before.

\smallskip

The supersaturation version of Sperner's problem is to determine the minimum number of pairs $F\subsetneq F'$ over all subfamilies of $2^{[n]}$ of given size. We say that a family $\cF$ is \textit{centered} if it consists of the sets closest to $n/2$. More precisely, if $F\in \cF$ and $||G|-n/2|<||F|-n/2|$ imply $G\in\cF$. Kleitman \cite{Kle66} proved that among families of cardinality $m$, centered ones contain the smallest number of copies of $P_2$. He conjectured that the same holds for any $P_k$. After several partial results, e.g. \cite{BalWag18,DasGanSud15,DovGriKan14}, the conjecture was confirmed by Samotij \cite{Sam19}. The following is a consequence of the result of Samotij. We will only use it with $k=2$, i.e. the result of Kleitman.

\begin{thm}\label{kchain}
For any $k,t$ with $k-1 \le t$ and $\varepsilon>0$ there exists $n_{k,t,\varepsilon}$ such that if $n\ge n_{k,t,\varepsilon}$, then any family $\cF\subseteq 2^{[n]}$ of size at least $(t+\varepsilon)\binom{n}{\lfloor n/2\rfloor}$ contains at least $\varepsilon \frac{n^t}{2^{t+1}}\binom{n}{\lfloor n/2\rfloor}$ chains of length $k$.
\end{thm}
Apart from the above, the only supersaturation result in the area of forbidden subposet problems is due to Patk\'os \cite{Pat15}. It gives the minimum number of copies of the butterfly poset\footnote{the poset on four elements $a,b<c,d$} $B$ in families of size $La(n,B)+E$ for small values of $E$.

We will investigate the number of copies of $P$ created when the number of additional sets compared to a largest $P$-free family is proportional to the size of the middle level $\binom{[n]}{\floor{n/2}}$. Let $M(n,P)$ denote the number of copies of $P$ in the $e(P)+1$ middle levels of $2^{[n]}$, and let $M^*(n,P)$ denote the number of induced copies of $P$ in the $e^*(P)+1$ middle levels of $2^{[n]}$. The \textit{Hasse diagram} of a poset $P$ is the directed graph with vertex set $P$ and for $p,q \in P$, $(pq)$ is an arc in the Hasse diagram if $p < q$ and there does not exist $z\in P$ with $p < z < q$. We say that $P$ is \textit{connected}, if its Hasse diagram (as a digraph) is weakly connected, i.e., we cannot partition its vertices into two sets such that there is no arc between those sets. The undirected Hasse diagram is the undirected graph obtained from the Hasse diagram by removing orientations of all arcs.

\begin{proposition}
For any connected poset $P$ on at least two elements there exist positive integers $x(P)$ and $x^*(P)$ such that $M(n,P)=\Theta\left(n^{x(P)}\binom{n}{\lfloor n/2\rfloor}\right)$ and $M^*(n,P)=\Theta\left(n^{x^*(P)}\binom{n}{\lfloor n/2\rfloor}\right)$ hold. 
\end{proposition}

\begin{proof}
The proofs of the two statements are analogous, so we include only that of the non-induced version. For a copy $\cG$ of $P$ with all sets from the $e(P)+1$ middle levels of $2^{[n]}$, let $A_\cG=\cap_{G\in \cG}G$, $B_\cG=\cup_{G\in \cG}G$ and $m_\cG=|A_\cG|, M_\cG=|B_\cG|$. Let us define $x(P)$ now. Let $x(P):=\max_\cG\{M_\cG-m_\cG\}$, where the maximum runs through all the copies $\cG$ of  $P$ with all sets from the $e(P)+1$ middle levels of $2^{[n]}$. 

We claim that for any such $\cG$, we have $M_\cG - m_\cG \le e(P)|P|$ (in other words $x(P) \le e(P)|P|$). Indeed, as $P$ is connected, we can go through its elements in an order such that every element is in relation with at least one of the earlier elements. As $\cG$ is from the $e(P)+1$ middle levels, this means that the new element is a set that is either contained in, or contains an earlier set, thus it differs from that set in at most $e(P)$ elements. In the first case, this new element decreases $m_\cG$ by at most $e(P)$, in the second case it increases $M_\cG$ by at most $e(P)$, so we are done. 

Similarly, one can show $|m_\cG-n/2|,|M_\cG-n/2|\le e(P)|P|$. Clearly, for any $A\subseteq B$ with $|B|-|A|\le e(P)|P|$ there is at most a fixed constant number of copies $\cG$ of $P$ such that $A=A_\cG$ and $B=B_\cG$. Finally, the number of pairs $A\subseteq B$, with $|B\setminus A|\le x(P)$, $||A|-n/2|\le e(P)|P|$ and $||B|-n/2|\le e(P)|P|$ is at most $C\binom{n}{\lfloor n/2\rfloor}\binom{\lfloor n/2\rfloor +e(P)|P|}{x(P)}$. This yields $M(n,P)=O(n^{x(P)}\binom{n}{\lfloor n/2\rfloor})$.

For the lower bound, fix a copy $\cG$ with $|B_\cG\setminus A_\cG|=x(P)$. Clearly, for any $A'\subseteq B'$ with $|A'|=|A_\cG|,|B'|=|B_\cG|$ there exists a permutation $\pi$ of $[n]$ with $\pi(A_\cG)=A'$ and $\pi(B_\cG)=B'$. Therefore, such permutations $\pi$ map $\cG$ into distinct copies of $P$.
Their number is clearly at least $\binom{n}{|A_\cG|}\binom{n-|A_\cG|}{|B_\cG\setminus A_\cG|}\ge c \cdot n^{x(P)}\binom{n}{\lfloor n/2\rfloor}$ for some positive constant $c$. 
\end{proof}

Now we can state the first generalization of Conjecture \ref{regisejtes}.

\begin{conjecture}\label{supersatsejtes} \ 

(i) For every poset $P$ and $\varepsilon>0$ there exists $\delta>0$ such that if $\cF\subseteq 2^{[n]}$ is of size at least $(e(P)+\varepsilon)\binom{n}{\lfloor n/2\rfloor}$, then $\cF$ contains at least $\delta\cdot M(n,P)$ many copies of $P$.

(ii) For every poset $P$ and $\varepsilon>0$ there exists $\delta>0$ such that if $\cF\subseteq 2^{[n]}$ is of size at least $(e^*(P)+\varepsilon)\binom{n}{\lfloor n/2\rfloor}$, then $\cF$ contains at least $\delta\cdot M^*(n,P)$ many induced copies of $P$.
\end{conjecture}

We will prove Conjecture \ref{supersatsejtes} for several classes of tree posets.  A poset $T$ is a \textit{tree poset}, if its undirected Hasse diagram is a tree. Note that for any tree poset $T$ of height 2, we have $x(T)=x^*(T)=|T|-1$.

\begin{thm}\label{height2}
Let $T$ be any height 2 tree poset of $t+1$ elements. Then for any $\varepsilon>0$ there exist $\delta>0$ and $n_0$ such that for any $n\ge n_0$ any family $\cF\subseteq 2^{[n]}$ of size $|\cF|\ge(1+\varepsilon)\binom{n}{\lfloor n/2\rfloor}$  contains at least $\delta n^t\binom{n}{\lfloor n/2\rfloor}$ copies of $T$.
\end{thm}

We say that a tree poset $T$ is \textit{upward (downward) monotone}, if for any $x\in T$ there exists at most 1 element $y \in T$ with $y\prec x$ ($x\prec y$). A tree poset is called \textit{monotone}, if it is either upward or downward monotone. 

\begin{thm}\label{monotone}
For any monotone tree poset $T$ and $\varepsilon>0$, there exist $\delta>0$ and $n_0$ such that for any $n\ge n_0$ any family $\cF\subseteq 2^{[n]}$ of size $|\cF|\ge (h(T)-1+\varepsilon)\binom{n}{\floor{n/2}}$ contains at least $\delta n^{x(T)}\binom{n}{\floor{n/2}}$ copies of $T$.
\end{thm}

The complete multipartite poset $K_{r_1,r_2,\dots,r_\ell}$ is a poset on $\sum_{i=1}^\ell r_i$ elements $a_{i,j}$ with $i=1,2,\dots,\ell$, $j=1,2,\dots,r_i$ such that $a_{i,j}<a_{i',j'}$ if and only if $i<i'$. The poset $K_{1,r}$ is usually denoted by $\vee_r$, and the poset $K_{r,1}$ is denoted by $\wedge_r$. The poset $K_{s,1,t}$ is a tree poset with $x(K_{s,1,t})=x^*(K_{s,1,t})=s+t$.

\begin{thm}\label{s1tsuper}
For any $s,t\in \mathbb{N}$ and $\varepsilon>0$ there exist $n_0=n_{\varepsilon,s,t}$ and $\delta>0$ such that any $\cF\subseteq 2^{[n]}$ of size at least $(2+\varepsilon)\binom{n}{\lfloor n/2\rfloor}$ with $n\ge n_0$ contains at least $\delta n^{s+t}\binom{n}{\lfloor n/2\rfloor}$ induced copies of $K_{s,1,t}$.
\end{thm}

We will consider the supersaturation problem for the generalized diamond $D_s$, i.e., the poset on $s+2$ elements with $a<b_1,b_2,\dots,b_s<c$. For any integer $s\ge 2$, let us define $m_s= \lceil\log_2(s + 2)\rceil$ and $m^*_s=\min\{m: s\le  \binom{m}{\lceil m/2\rceil}\}$. Clearly, for any integer $s\ge 2$, we have $e(D_s)=x(D_s)=m_s$ and $e^*(D_s)=x^*(D_s)=m^*_s$. The next theorem establishes a lower bound that is less by a factor of $\sqrt{n}$
than what Conjecture \ref{supersatsejtes} states for diamond posets $D_s$ for infinitely many $s$.

\begin{thm}\label{diamond} \

(i) If $s \in [2^{m_s - 1} -1,2^{m_s} - \binom{m_s}{\lceil \frac{m_s}{2}\rceil}-1]$, then for any $\varepsilon>0$ there exists a $\delta>0$ such that every $\cF\subseteq 2^{[n]}$ with $|\cF|\ge (m_s+\varepsilon)\binom{n}{\lfloor n/2 \rfloor}$ contains at least $\delta \cdot n^{m_s-0.5}\binom{n}{\lfloor n/2 \rfloor}$ copies of $D_s$.

(ii) For any $\varepsilon>0$ there exists a $\delta>0$ such that every $\cF\subseteq 2^{[n]}$ with $|\cF|\ge (4+\varepsilon)\binom{n}{\lfloor n/2\rfloor}$ contains at least $\delta \cdot n^{3.5}\binom{n}{\lfloor n/2 \rfloor}$ induced copies of $D_4$. 

(iii) For any constant $c$ with $1/2<c<1$ there exists an integer $s_c$ such that if $s \ge s_c$ and $s\le c\binom{m^*_s}{\lfloor m^*_s/2\rfloor}$, then the following holds: for any $\varepsilon>0$  there exists a $\delta>0$ such that every $\cF\subseteq 2^{[n]}$ with $|\cF|\ge (m^*_s+\varepsilon)\binom{n}{\lfloor n/2\rfloor}$ contains at least $\delta \cdot n^{m^*_s-0.5}\binom{n}{\lfloor n/2 \rfloor}$ induced copies of $D_s$.
\end{thm}

Let us turn our attention to counting (induced) $P$-free families. 
As we mentioned earlier, every subfamily of a $P$-free family is $P$-free, therefore $2^{La(n,P)}\ge 2^{(e(P)+o(1))\binom{n}{\floor{n/2}}}$ is a lower bound on the number of such families. Determining the number of $P_2$-free families has attracted a lot of attention. The upper bound $2^{(1+o(1))\binom{n}{\floor{n/2}}}$, asymptotically matching in the exponent the trivial lower bound was obtained by Kleitman \cite{Kle69}. After several improvements, Korshunov \cite{Kor81} determined asymptotically the number of $P_2$-free families.

\begin{conjecture}\label{leszamlsejtes}
(i) The number of $P$-free families in $2^{[n]}$ is $2^{(e(P)+o(1))\binom{n}{\lfloor n/2\rfloor}}$.

(ii) The number of induced $P$-free families in $2^{[n]}$ is $2^{(e^*(P)+o(1))\binom{n}{\lfloor n/2\rfloor}}$.
\end{conjecture}


\begin{thm}\label{s1tcount}
(i) The number of induced $\vee_r$-free families is $2^{(1+o(1))\binom{n}{\floor{n/2}}}$.

(ii) The number of induced $K_{s,1,t}$-free families in $2^{[n]}$ is $2^{(2+o(1))\binom{n}{\floor{n/2}}}$.
\end{thm}

As every height 2 poset $P$ is a non-induced subposet of $K_{|P|,1,|P|}$, Conjecture \ref{leszamlsejtes} (i) is an immediate consequence of Theorem \ref{s1tcount} for those height 2 posets $P$ for which $e(P)=2$.

\medskip

Finally, we turn to random versions of forbidden subposet problems. The probabilistic version of Sperner's theorem was proved by Balogh, Mycroft, and Treglown \cite{BalMycTre14} and Collares and Morris \cite{ColMor16}, independently. It states that if $p=\omega(1/n)$, then the largest antichain in $\cP(n,p)$ is of size $(1+o(1))p\binom{n}{\floor{n/2}}$ w.h.p.. This is sharp in the sense that if $p=o(1/n)$ then the asymptotics is different. Note that as any $k$-Sperner family is the union of $k$ antichains, the analogous statement holds for $k$-Sperner families in $\cP(n,p)$. Both papers used the container method. Hogenson in her PhD thesis \cite{viking} adapted the method of Balogh, Mycroft, and Treglown to obtain the same results for non-induced $\vee_r$-free families.

Let us state a general proposition that gives a range of $p$ when one can have a $P$-free family in $\cP(n,p)$ that is larger than $p(e(P)+o(1))\binom{[n}{\floor{n/2}}$.

\begin{proposition}
For any finite connected poset $P$, the following statements hold.

(i) If $p=o(n^{-\frac{x(P)}{|P|-1}})$, then the largest $P$-free family in $\cP(n,p)$ has size at least $(e(P)+1-o(1))p\binom{n}{\floor{n/2}}$ w.h.p..

(ii) If $p=o(n^{-\frac{x^*(P)}{|P|-1}})$, then the largest induced $P$-free family in $\cP(n,p)$ has size at least $(e^*(P)+1-o(1))p\binom{n}{\floor{n/2}}$ w.h.p..
\end{proposition}

\begin{proof} We only prove $(i)$, the proof of $(ii)$ is similar.
Let us denote the random family of the $e(P)+1$ middle levels after keeping any of its sets with probability $p$ by $\cM_p=\cM \cap \cP(n,p)$. Clearly, $\mathbb{E}(|\cM_p|)=(e(P)+1+o(1))p\binom{n}{\floor{n/2}}$ and thus we have $|\cM_p|=(e(P)+1+o(1))p\binom{n}{\floor{n/2}}$ w.h.p.. Let $X$ be the random variable that denotes the number of copies of $P$ in $\cM_p$. Then we have $\mathbb{E}(X)=\Theta(p^{|P|}n^{x(P)}\binom{n}{\floor{n/2}})$. By the assumption on $p$, we have $p^{|P|}n^{x(P)}\binom{n}{\floor{n/2}}=o(p\binom{n}{\floor{n/2}})$, and so $X=o(|\cM_p|)$ w.h.p., and thus by removing the copies of $P$ from $\cM_p$, we obtain a $P$-free family in $\cP(n,p)$ of size $(e(P)+1-o(1))p\binom{n}{\floor{n/2}}$ w.h.p..
\end{proof}

If $\cM_p$ does not contain a subposet $P'$ of $P$, then it is $P$-free, thus we have the following.

\begin{cor}
For any finite poset $P$, let $d(P)=\min \frac{x(P')}{|P'|-1}$, where $P'$ runs through all connected subposets $P'$ of $P$ with $e(P)=e(P')$. Similarly, let $d^*(P)=\min \frac{x^*(P')}{|P'|-1}$, where $P'$ runs through all connected subposets $P'$ of $P$ with $e^*(P)=e^*(P')$.  Then the following statements hold.

\smallskip 

(i) If $p=o(n^{-d(P)})$, then the largest $P$-free family in $\cP(n,p)$ has size at least $(e(P)+1-o(1))p\binom{n}{\floor{n/2}}$ w.h.p.

\smallskip 

(ii) If $p=o(n^{-d^*(P)})$, then the largest induced $P$-free family in $\cP(n,p)$ has size $(e^*(P)+1-o(1))p\binom{n}{\floor{n/2}}$ w.h.p.
\end{cor}

We conjecture that the bounds above are sharp.

\begin{conjecture}
For any finite connected poset $P$ the following statements hold.

\smallskip 

\noindent (i) If $p=\omega(n^{-d(P)})$, then the largest $P$-free family in $\cP(n,p)$ has size $(e(P)+o(1))p\binom{n}{\floor{n/2}}$ w.h.p..

\smallskip 

\noindent (ii) If $p=\omega(n^{-d^*(P)})$, then the largest induced $P$-free family in $\cP(n,p)$ has size $(e^*(P)+o(1))p\binom{n}{\floor{n/2}}$ w.h.p..
\end{conjecture}

\begin{thm}\label{easy}
 If $p=\omega(1/n)$, then the following are true.

\smallskip 

\noindent    (i) For any integer $r\ge 0$, the largest induced $\vee_{r+1}$-free family in $\cP(n,p)$ has size $(1+o(1))p\binom{n}{\lfloor n/2\rfloor}$ w.h.p..

\smallskip 

\noindent (ii)    For any pair $s,t\ge 1$ of integers, the largest induced $K_{s,1,t}$-free family in $\cP(n,p)$ has size $(2+o(1))p\binom{n}{\lfloor n/2\rfloor}$ w.h.p.. 
\end{thm}

The structure of the paper is as follows: in the next section, we gather some earlier results that will be used as tools in the proofs of our theorems. Also, we will obtain an induced $\vee_{r+1}$-free container lemma. Section 3, in three subsections, contain the proofs of our theorems.

\section{Preliminaries}

Here and in the next section, we will assume that $n$ is large enough whenever it is necessary.

\begin{lemma}\label{binom}
For any $1 \le l\le n/2$ we have $\sum_{i=0}^{l-1}\binom{n}{i}\le 2\sqrt{n}\binom{n}{l}$.
\end{lemma}

\begin{proof}
Let $m=\lfloor \sqrt{n}\rfloor$ and observe that for any $k<n/2$ we have 
\[
\frac{\binom{n}{k-m}}{\binom{n}{k}}\le \frac{\binom{n}{\lfloor n/2\rfloor-m}}{\binom{n}{\lfloor n/2\rfloor}}\le  \prod_{i=1}^m\frac{\lfloor n/2\rfloor-i+1}{\lceil n/2\rceil+i}\le e^{-\sum_{i=1}^m\frac{2i}{n}}\le e^{-1}.
\]
So dividing $\sum_{i=0}^{l-1}\binom{n}{i}$ into $m$ subsums depending on the residue of $i$ mod $m$, we obtain subsums that can be upper bounded by geometric progressions of quotient $e^{-1}$.
\end{proof}

We denote by $\bC_k$ the set of chains of length $k+1$, i.e. the set of maximal chains in $2^{[k]}$.

\begin{lemma}[Griggs, Li, Lu, in the proof of Theorem 2.5 in \cite{GriLiLu12}]
\label{nonindlem} \ 

If $s \in [2^{m_s - 1} -1,2^{m_s} - \binom{m_s}{\lceil \frac{m_s}{2}\rceil}-1]$, and $\cG \subseteq 2^{[k]}$ is a $D_{s}$-free family of sets, then the number of pairs $(G,\cC)$ with $G \in \cG\cap \cC$ and $\cC \in \bC_k$ is at most $m_sk!$.
\end{lemma}

\begin{lemma}[Patk\'os \cite{Pat15b}]\label{plem} \

(i) Let $\cG\subseteq 2^{[k]}$ be a family of sets such that any antichain $\cA \subset \cG$ has size at most 3. Then the number of pairs $(G,\cC)$ with $G \in \cG \cap \cC$ and $\cC \in \bC_k$ is at most $4k!$.

(ii) For any constant $c$ with $1/2<c<1$ there exists an integer $s_c$ such that if $s \ge s_c$ and $s\le c\binom{m^*_s}{\lceil m^*_s/2\rceil}$, then the following holds: if $\cG\subseteq 2^{[k]}$ is a family of sets such that any antichain $\cA \subset \cG$ has size less than $s$, then the number of pairs $(G,\cC)$ with $G \in \cG \cap \cC$ and $\cC \in \bC_k$ is at most $m^*_sk!$.
\end{lemma}

The special case $r=1$ of the next theorem was proved by Carroll and Katona \cite{CarKat08}. 

\begin{thm}[Griggs, Li \cite{GriLi13}]\label{carkat} \

If $\cF\subseteq 2^{[n]}$ is an induced $\vee_{r+1}$-free family, then $|\cF|\le (1+\frac{2r}{n}+O(\frac{r}{n^2}))\binom{n}{\lfloor n/2\rfloor}$ holds.
\end{thm}

We remark that we will use the above theorem with $r=\Theta(n)$. 

\begin{thm}\label{induceddegree} \ 

For every $\varepsilon>0$ and $t,r\in \mathbb{Z}^+$, if $n$ is large enough and $\cF\subseteq 2^{[n]}$ is of size at least $(t+\varepsilon)\binom{n}{\floor{n/2}}$, then there exists $F\in \cF$ such that $F$ is the bottom element of an induced copy of $\vee_{\delta n^t}$ with $\delta=\frac{\varepsilon}{2^{t+3}t!}$. 
\end{thm}

\begin{proof}
As $|\{G\subseteq [n]: ||G|-n/2|\ge n^{2/3}\}|\le \frac{1}{n^2}\binom{n}{\floor{n/2}}$, we can assume that $n/2-n^{2/3}\le |F|\le n+n^{2/3}$ holds for every $F\in \cF$ and $|\cF|\ge (t+\varepsilon/2)\binom{n}{\floor{n/2}}$. The Lubell-function $\lambda_n(\cF):=\sum_{F\in \cF}\frac{1}{\binom{n}{|F|}}$ of a family $\cF\subseteq 2^{[n]}$ of sets is the average number of sets in $\cF\cap \cC$ over all maximal chains $\cC\in \bC_n$. Clearly, $|\cF|\ge (t+\varepsilon/2)\binom{n}{\floor{n/2}}$ implies $\lambda_n(\cF)\ge t+\varepsilon/2$. 

Let $\bC_F$ be the collection of maximal chains $\cC\in \bC_n$ such that $F$ is the smallest set in $\cF\cap \cC$, and let $\bC_\emptyset$ be the collection of those maximal chains that avoid $\cF$.
Let us partition $\bC_n$ into $\bC_\emptyset~\cup~\bigcup_{F\in \cF}\bC_F$. Writing $\cF_F=\{F'\setminus F: F'\in \cF, F\subseteq F'\}$, we obtain $\lambda_n(\cF)=\sum_{F\in \cF}\frac{|\bC_F|}{n!}\lambda_{n-|F|}(\cF_F)$. This means that for some $F\in \cF$, we must have $\lambda_{n-|F|}(\cF_F)\ge t+\varepsilon/2$. If $\cF$ does not contain any induced copy of $\vee_{\delta n^t}$, then for any $i\ge t$ we have $|\{F'\in \cF_F:|F'|=i\}|<\delta n^t$, thus writing $\cF_{F,1}=\{G\in \cF_F:|G|\le t-1\}$ and $\cF_{F,2}=\cF_F\setminus\cF_{F,1}$ we obtain $$t+\varepsilon/2\le \lambda_{n-|F|}(\cF_F)=\lambda_{n-|F|}(\cF_{F,1})+\lambda_{n-|F|}(\cF_{F,2})\le$$ 
$$ \le  t+\delta n^t\sum_{i=t}^{2n^{2/3}}\frac{1}{\binom{n-|F|}{i}} < t+\frac{2\delta n^t}{\binom{n/2-n^{2/3}}{t}}\le t+2^{t+2}\delta t!.$$
This is a contradiction if $\delta\le\frac{\varepsilon}{2^{t+3}t!}$.
\end{proof}

Hogenson \cite{viking}, altering a proof of Balogh, Mycroft, and Treglown \cite{BalMycTre14},  obtained a container lemma for non-induced $\vee_{r+1}$-free families. With the help of Theorem \ref{induceddegree}, we can validate it for induced $\vee_{r+1}$-free families. 

For a family $\cG$ of sets in $2^{[n]}$ and a set $F \subseteq 2^{[n]}$, we introduce $U_\cG(F)=\{G\in \cG:F\subseteq G\}$. Previous proofs used the $\cG$-\textit{degree} of $F$, the size of $U_\cG(F)$. Observe that $\cF$ is $\vee_{r+1}$-free if and only if the $\cF$-degree of every $F\in \cF$ is at most $r$. We introduce the $\cG$-\textit{weight} of $F$ as the size of the largest antichain in $U_\cG(F)$. Observe that $\cF$ is induced $\vee_{r+1}$-free if and only if the $\cF$-weight of every $F\in \cF$ is at most $r$. Now we state the appropriate container lemma.

\begin{thm}\label{inducedcontainer}
Let $t,r\in \mathbb{Z}^+$ and $\varepsilon \le \frac{1}{(2t)^{t+1}}$ and assume  $n$ is large enough. Then there exist functions $f:\binom{2^{[n]}}{\le (r+1)2^nn^{-(t+0.9)}}\rightarrow \binom{2^{[n]}}{(t+1+\varepsilon)\binom{n}{\floor{n/2}}}$ and $g:\binom{2^{[n]}}{\le (r+1)\frac{t+2}{\varepsilon^2n^t}\binom{n}{\floor{n/2}}} \rightarrow \binom{2^{[n]}}{(t+\varepsilon)\binom{n}{\floor{n/2}}}$ such that for any induced $\vee_{r+1}$-free family $\cF\subseteq 2^{[n]}$ there exist disjoint subfamilies $\cH_1,\cH_2\subseteq \cF$ such that $(\cH_1\cup \cH_2)\cap g(\cH_1\cup \cH_2)=\emptyset$, $\cH_2\subseteq f(\cH_1)$ and $\cF\subseteq \cH_1\cup\cH_2\cup g(\cH_1\cup \cH_2)$.
\end{thm}

\begin{proof}
The proof uses the standard graph container algorithm with some modifications. First, we fix an ordering $S_1,S_2,\dots,S_{2^n}$ of $2^{[n]}$ and also an ordering $\cS_1,\cS_2,\dots,\cS_{2^{2^n}}$ of $2^{2^{[n]}}$. The input of the algorithm is an induced $\vee_{r+1}$-free family $\cF\subseteq 2^{[n]}$, and in two phases it outputs $\cH_1$, $\cH_2$, $f(\cH_1)$ and $g(\cH_1\cup \cH_2)$ as follows. 

\medskip 

At the beginning we set $\cG^0=2^{[n]}$, $\cH^0_1=\cH^0_2=\emptyset$ and we start at Phase I. 

\smallskip 

Later in the $i$th round (for $i=1,2,\dots$) we always pick a $G_i\in \cG^{i-1}$ with largest $\cG^{i-1}$-weight. If there are several sets $G$ with the same (largest) $\cG^{i-1}$-weight, we pick the one appearing first in our fixed ordering of $2^{[n]}$. 

\bigskip  

\textbf{Phase I.}

\medskip  

$\bullet$ If $G_i\notin \cF$, we set $\cG^{i}=\cG^{i-1}\setminus\{G\}$ and $\cH^{i}_1=\cH^{i-1}_1$, $\cH^{i}_2=\cH^{i-1}_2$.

\smallskip 

$\bullet$ If $G_i\in \cF$, and the $\cG^{i-1}$-weight of $G_i$ is more than $n^{t+0.9}$, then we pick the largest antichain $\cA^{i}$ in $U_{\cG^{i-1}}(G_i)$. If there are multiple such antichains, we pick the one with the smallest index in our fixed ordering of $2^{2^{[n]}}$ and set $\cG^{i}=\cG^{i-1}\setminus (\cA^{i} \cup \{G_i\})$, $\cH^{i}_1=\cH^{i-1}_1\cup [(\cA^{i}\cap \cF)\cup \{G_i\}]$ and $\cH^{i}_2=\cH^{i-1}_2$.

\smallskip 

$\bullet$ If $G_i\in \cF$, and the $\cG^{i-1}$-weight of $G_i$ is at most $n^{t+0.9}$, then Phase I is ended, we keep $\cG^{i}=\cG^{i-1}$, $\cH^{i}_2=\cH^{i-1}_2$, set $\cH_1=\cH^{i-1}_1$ (from here on, $\cH^{i}_1$ does not change) and define $f(\cH_1)=\cG^{i}$.  And jump to Phase II.

\medskip  

\textbf{Phase II.}

\medskip  

$\bullet$ If $G_i\notin \cF$, we set $\cG^{i}=\cG^{i-1}\setminus\{G\}$ and $\cH^{i}_1=\cH^{i-1}_1$, $\cH^{i}_2=\cH^{i-1}_2$.

\smallskip 

$\bullet$ If $G_i\in \cF$, and the $\cG^{i-1}$-weight of $G_i$ is more than $\varepsilon^2n^{t}$, then we again take the largest antichain $\cA^{i}$ in $U_{\cG^{i-1}}(G_i)$ with the smallest index in our fixed ordering of $2^{2^{[n]}}$ and set $\cG^{i}=\cG^{i-1}\setminus (\cA^{i} \cup \{G_i\})$, $\cH^{i}_2=\cH^{i-1}_2\cup [(\cA^{i}\cap \cF)\cup \{G_i\}]$.

\smallskip 

$\bullet$ If $G_i\in \cF$, and the $\cG^{i-1}$-weight of $G_i$ is at most $\varepsilon^2n^{t}$, then Phase II and the algorithm is ended, set $\cH_2=\cH^{i-1}_2$  and define $g(\cH_1\cup \cH_2)=\cG^{i-1}$.

\bigskip 

Observe the following:
\begin{itemize}
    \item 
    whenever we include sets in $\cH^{i}_1$, then the number of such sets is at most $r+1$ (as $\cF$ is \newline $\vee_{r+1}$-free), and the number of sets removed from $\cG^{i-1}$ is at least $n^{t+0.9}$, so $$|\cH_1|\le(r+1)2^n/n^{t+0.9};$$
    \item
    at the end of Phase I, $\cG^{i}$ does not contain any induced copies of $\vee_{n^{t+0.9}}$, so, by Theorem~\ref{induceddegree}, $$|f(\cH_1)|=|\cG^{i}|\le (t+1+\varepsilon)\binom{n}{\floor{n/2}};$$
    \item
    the above two bullet points and the threshold for Phase II imply that $$|\cH_2|\le \frac{r+1}{\varepsilon^2n^t}|f(\cH_1)|\le  (r+1)\frac{t+2}{\varepsilon^2n^t}\binom{n}{\floor{n/2}};$$
    \item
    Theorem \ref{induceddegree} implies that at the end of Phase II,  $$|g(\cH_1\cup \cH_2)|=|\cG^{i}|\le (t+\varepsilon)\binom{n}{\floor{n/2}}.$$

\end{itemize}

All what remains to prove is that the function $f$ and $g$ are well defined, i.e., if for two distinct $\vee_{r+1}$-free families $\cF$ and $\cF'$ the algorithm outputs the same $\cH_1$, then $f(\cH_1)$ is defined the same, and if in addition $\cH_1\cup \cH_2$ is the same for both run of the algorithm, then so is $g(\cH_1\cup \cH_2)$. We claim more: the families $\cG^{i},\cH^{i}_1,\cH^{i}_2$ are the same for both runs for all values of $i=0,1,\dots$. This is certainly true for $i=0$. Then by induction, if this holds for some $i$, then $G^{i}$ is the same for both run. Therefore, due to the fixed ordering of $2^{[n]}$, the algorithm considers the same set $G_{i+1}\in \cG^{i}$ in step $i+1$. As $G_{i+1}$ will be removed from $\cG^{i}$ in all cases, therefore it cannot happen that $G_{i+1}$ belongs to exactly one of $\cF,\cF'$. If $G_{i+1}\notin \cF,\cF'$, then $G_{i+1}$ is removed from $\cG^{i}$and nothing happens to $\cH_1,\cH_2$ so the claim is true for $i+1$. If $G_{i+1}\in \cF\cap \cF'$, then, due to the fixed ordering of $2^{2^{[n]}}$, the antichain $\cA^{i+1}$ is defined the same for both runs. This immediately yields that $\cG^{i+1}$ is defined the same for both runs. Also, as for any $A\in \cA^{i+1}$ either $A$ becomes a set of $\cH_1,\cH_2$ in this step or never, and in the end these sets are the same, therefore they must be the same after step $i+1$.
\end{proof}

\section{Proofs}

\subsection{$\vee_{r+1}$-free families and consequences - Theorem \ref{s1tsuper}, \ref{s1tcount}, and \ref{easy}}

In this subsection we present the proofs of our theorems concerning $\vee_{r+1}$-free and $K_{s,1,t}$-free families. We restate the theorems here for convenience.

\begin{thmn}[\ref{s1tsuper}]
For any $s,t\in \mathbb{N}$ and $\varepsilon>0$ there exist $n_0=n_{\varepsilon,s,t}$ and $\delta>0$ such that any $\cF\subseteq 2^{[n]}$ of size at least $(2+\varepsilon)\binom{n}{\lfloor n/2\rfloor}$ with $n\ge n_0$ contains at least $\delta n^{s+t}\binom{n}{\lfloor n/2\rfloor}$ induced copies of $K_{s,1,t}$.
\end{thmn}

\begin{proof}
Let $\cF\subseteq 2^{[n]}$ be a family of sets of size $(2+\varepsilon)\binom{n}{\floor{n/2}}$.  
Let $\cD$ be the family of those elements of $\cF$ that are not the maximal element of an induced $\wedge_{\varepsilon n/10}$ and $\cU$ the family of those that are not the minimal element of an induced $\vee_{\varepsilon n/10}$.
By Theorem \ref{carkat}, we have $|\cD|,|\cU|\le (1+4\varepsilon /10)\binom{n}{\floor{n/2}}$, thus $|\cF\setminus (\cD\cup \cU)|\ge \frac{\varepsilon}{5}\binom{n}{\floor{n/2}}$. Taking sets from $\cF\setminus (\cD\cup\cU)$ to play the role of the middle element of $K_{s,1,t}$, by definition of $\cD$ and $\cU$, we obtain at least $\frac{\varepsilon}{5}\binom{n}{\floor{n/2}}\binom{\varepsilon n/10}{s}\binom{\varepsilon n/10}{t}$ copies of $K_{s,1,t}$.
\end{proof}

\begin{thmn}[\ref{s1tcount}] \ 

(i) The number of induced $\vee_r$-free families is $2^{(1+o(1))\binom{n}{\floor{n/2}}}$.

(ii) The number of induced $K_{s,1,t}$-free families in $2^{[n]}$ is $2^{(2+o(1))\binom{n}{\floor{n/2}}}$.
\end{thmn}

\begin{proof}
To prove (i), we apply our container lemma, Theorem \ref{inducedcontainer} with $t=1$. It shows that for every induced $\vee_{r+1}$-free family $\cF$, there exist $\cH_1,\cH_2$ and $g(\cH_1\cup \cH_2)$ such that $\cH_1$ and $\cH_2$ are disjoint, $|\cH_1\cup \cH_2|\le (r+1)\frac{3}{\varepsilon^2n}\binom{n}{\floor{n/2}}$, $|g(\cH_1\cup \cH_2)|\le (1+\varepsilon)\binom{n}{\floor{n/2}}$ and $\cF\subseteq \cH_1\cup \cH_2\cup g(\cH_1\cup \cH_2)$. Therefore, $|\cH_1\cup\cH_2\cup g(\cH_1\cup \cH_2)|\le (1+2\varepsilon)\binom{n}{\floor{n/2}}$. The number of subfamilies of such containers is at most $2^{(1+2\varepsilon)\binom{n}{\floor{n/2}}}$, and the number of such containers is at most $$\binom{2^n}{(r+1)\frac{3}{\varepsilon^2n}\binom{n}{\floor{n/2}}}2^{(r+1)\frac{3}{\varepsilon^2n}\binom{n}{\floor{n/2}}}.$$ Indeed, the first term is an obvious upper bound on the number of families $\cH_1\cup \cH_2$, and the second term is an obvious upper bound on the number of ways to partition $\cH_1\cup \cH_2$ to $\cH_1$ and $\cH_2$.
Using $\binom{n}{\floor{n/2}}=\Theta(\frac{1}{\sqrt{n}}2^n)$ and $\binom{a}{b}\le (\frac{ea}{b})^b$, we obtain that the number of induced $\vee_{r+1}$-free families in $2^{[n]}$ is at most
$$\binom{2^n}{(r+1)\frac{3}{\varepsilon^2n}\binom{n}{\floor{n/2}}}2^{(r+1)\frac{3}{\varepsilon^2n}\binom{n}{\floor{n/2}}}2^{(1+2\varepsilon)\binom{n}{\floor{n/2}}}=2^{(1+2\varepsilon+O_{\varepsilon,r}(\frac{\log n}{n}))\binom{n}{\floor{n/2}}}.$$
The lower bound follows from the fact that every subfamily of the middle level is $\vee_r$-free.

To prove (ii), observe that every induced $K_{s,1,t}$-free family $\cF\subseteq 2^{[n]}$ can be written as $\cF=\cD\cup \cU$ such that $\cD$ is induced $\wedge_s$-free and $\cU$ is induced $\vee_t$-free. Indeed, 
let $\cD$ be the family of those elements of $\cF$ that are not the maximal element of an induced $\wedge_{s}$ and $\cU$ be the family of those that are not the minimal element of an induced $\vee_{t}$.
If $F\in \cF$ does not belong to $\cD\cup \cU$, then $F$ with its $s$ subsets and its $t$ supersets form an induced $K_{s,1,t}$, which contradicts the $K_{s,1,t}$-free property of $\cF$. By part (i), there are $2^{(1+o(1))\binom{n}{\lfloor n/2\rfloor}}$ induced $\vee_t$-free families in $2^{[n]}$ and there are $2^{(1+o(1))\binom{n}{\lfloor n/2\rfloor}}$ induced $\wedge_s$-free families in $2^{[n]}$. Therefore there are at most $2^{(2+o(1))\binom{n}{\lfloor n/2\rfloor}}$ induced $K_{s,1,t}$-free families in $2^{[n]}$. The lower bound immediately follows from the fact that every subfamily of the middle two levels is $K_{s,1,t}$-free.
\end{proof}

\begin{thmn}[\ref{easy}]
If $p=\omega(1/n)$, then the following are true.

\noindent    (i) For any integer $r\ge 0$, the largest induced $\vee_{r+1}$-free family in $\cP(n,p)$ has size $(1+o(1))p\binom{n}{\lfloor n/2\rfloor}$ w.h.p.

\noindent (ii)    For any pair $s,t\ge 1$ of integers, the largest induced $K_{s,1,t}$-free family in $\cP(n,p)$ has size $(2+o(1))p\binom{n}{\lfloor n/2\rfloor}$ w.h.p. 
\end{thmn}

\begin{proof} To prove (i), we again apply Theorem \ref{inducedcontainer} with $t=1$. All calculations are very close to those in \cite{viking}, which in turn are almost the same as those in \cite{BalMycTre14}, we include them for sake of completeness.  

It is enough to prove the statement for $\varepsilon <\frac{1}{4}$ and set $\varepsilon_1=\eps/4$. We will show that w.h.p. for every $\vee_{r+1}$-free family $\cF$ in $2^{[n]}$ of size at least $(1+\varepsilon)p\binom{n}{\floor{n/2}}$, not all sets of $\cF$ remain in $\cP(n,p)$. For every such $\cF$, Theorem \ref{inducedcontainer} with $\eps_1$ in the role of $\eps$, gives us $\cH_1=\cH_1(\cF),\cH_2=\cH_2(\cF)$ such that 
\begin{itemize}
\item
$\cH_1\in \binom{2^{[n]}}{\le (r+1)2^nn^{-1.9}}$; therefore the number of possible $\cH_1$'s is at most $$\sum_{a \le (r+1)2^nn^{-1.9}}\binom{2^n}{a}.$$ Clearly, we have $\mathbb{P}(\cH_1\subseteq \cP(n,p))=p^{|\cH_1|}$.
\item
$\cH_2\in \binom{2^{[n]}}{\le 3(r+1)\binom{n}{\floor{n/2}}/(\varepsilon_1^2n)}$ and $\cH_2\subseteq f(\cH_1) \in \binom{2^{[n]}}{(2+\varepsilon_1)\binom{n}{\floor{n/2}}}$, so for fixed $\cH_1$ the number of possible $\cH_2$'s is at most 
$$\left|\binom{f(\cH_1)}{\le 3(r+1)\binom{n}{\floor{n/2}}/(\varepsilon_1^2n)}\right|\le \sum_{b\le 3(r+1)\binom{n}{\floor{n/2}}/(\varepsilon_1^2n)}\binom{3\binom{n}{\floor{n/2}}}{b}.$$ Also,  $\mathbb{P}(\cH_2\subseteq \cP(n,p))=p^{|\cH_2|}$.
\item
For fixed $\cH_1$ and $\cH_2$ the corresponding $\cF$'s are all subfamilies of $\cH_1\cup \cH_2 \cup g(\cH_1\cup \cH_2)$ and contain $\cH_1\cup \cH_2$. 

\smallskip 

1. Let  $\cE_{\cH_1,\cH_2}$ be the event that there exists \textit{any} $\cF$ with $\cH_1(\cF)=\cH_1$, $\cH_2(\cF)=\cH_2$, and  $|\cF|\ge (1+\varepsilon)p\binom{n}{\floor{n/2}}$, and 

2. let $\cE_{g(\cH_1\cup \cH_2)}$ be the event that $|\cP(n,p)\cap g(\cH_1\cup \cH_2)|\ge (1+\varepsilon)p\binom{n}{\floor{n/2}}-|\cH_1\cup \cH_2|$ holds.

\smallskip 

We bound the probability of the event $\cE_{\cH_1,\cH_2}$ by the probability of the event $\cE_{g(\cH_1\cup \cH_2)}$. Note that 
$$(1+\varepsilon)p\binom{n}{\floor{n/2}}-|\cH_1\cup \cH_2|\ge (1+\varepsilon/2)p\binom{n}{\floor{n/2}}$$ and $$|g(\cH_1\cup \cH_2)|\le (1+\eps_1)\binom{n}{\floor{n/2}}\le (1+\eps/4)\binom{n}{\floor{n/2}}.$$ 
Therefore, $|\cP(n,p)\cap g(\cH_1\cup \cH_2)|$ is binomially distributed with $$\mathbb{E}(|\cP(n,p)\cap g(\cH_1\cup \cH_2)|)\le (1+\eps/4)p\binom{n}{\floor{n/2}},$$ so by Chernoff's inequality we have $$\mathbb{P}(\cE_{g(\cH_1,\cH_2)})\le\mathbb{P}\left[|\cP(n,p)\cap g(\cH_1\cup \cH_2)|\ge (1+\varepsilon/2)p\binom{n}{\floor{n/2}}\right]\le e^{-\eps^2p\binom{n}{\floor{n/2}}/100}.$$
\end{itemize}

Note that $\cH_1$, $\cH_2$, and $g(\cH_1\cup \cH_2)$ are disjoint, so the three events that $\cH_1\subseteq \cP(n,p)$, $\cH_2\subseteq \cP(n,p)$ and $\cE_{g(\cH_1\cup \cH_2)}$ are independent. Hence the probability that for fixed $\cH_1$ and $\cH_2$ there is a corresponding large induced $\vee_{r+1}$-free family $\cF$, is at most $p^{|\cH_1|+|\cH_2|}e^{-\eps^2p\binom{n}{\floor{n/2}}/100}$. Summing up for all possible $\cH_1$ and $\cH_2$, we obtain that the probability $\Pi$, that there is an induced $\vee_{r+1}$-free family in $\cP(n,p)$ of size $(1+\varepsilon)p\binom{n}{\floor{n/2}}$, is at most
\[\sum_{0\le a \le (r+1)n^{-1.9}2^n}\sum_{0 \le b \le 3(r+1)\binom{n}{\floor{n/2}}/(\varepsilon_1^2n)}\binom{2^n}{a}p^a\binom{3\binom{n}{\floor{n/2}}}{b}p^be^{-\eps^2p\binom{n}{\floor{n/2}}/100}.
\]
It is not hard to verify  that the largest summand in the above sum belongs to the largest possible values of $a$ and $b$. Therefore the above expression is bounded from above by
\[
  (r+1)n^{-1.9}2^n\frac{3(r+1)\binom{n}{\floor{n/2}}}{\varepsilon_1^2n}\binom{2^n}{(r+1)n^{-1.9}2^n}\binom{3\binom{n}{\floor{n/2}}}{\frac{3(r+1)\binom{n}{\floor{n/2}}}{(\varepsilon_1^2n)}}\
  \cdot e^{-\eps^2p\binom{n}{\floor{n/2}}/100}p^{{(r+1)}n^{-1.9}2^n}p^{3(r+1)\binom{n}{\floor{n/2}}/(\varepsilon^2n)}.
\]

Observe that $$(r+1)n^{-1.9}2^n3(r+1)\binom{n}{\floor{n/2}}/(\varepsilon_1^2n)\le e^{O(n)}\le e^{\eps^2p\binom{n}{\floor{n/2}}/400}.$$ 
Also, using $\binom{n}{k}\le (\frac{en}{k})^k$ and $p\binom{n}{\floor{n/2}}=\omega(n^{-1.5}2^n)$, we have
\[
\binom{2^n}{(r+1)n^{-1.9}2^n}p^{{(r+1)}n^{-1.9}2^n} \le (en^{1.9}p)^{(r+1)n^{-1.9}2^n}
 \le e^{{O(n^{-1.9}2^n\ln n)}}\le e^{\eps^2p\binom{n}{\floor{n/2}}/400}.
\]

Finally, by the same reasoning we have
\begin{equation*}
\begin{split}
\binom{3\binom{n}{\floor{n/2}}}{3(r+1)\binom{n}{\floor{n/2}}/(\varepsilon_1^2n)}p^{3(r+1)\binom{n}{\floor{n/2}}/(\varepsilon_1^2n)} & \le (e\varepsilon_1^2np)^{3(r+1)\binom{n}{\floor{n/2}}/(\varepsilon_1^2n)}\\
& \le e^{3(r+1)\binom{n}{\floor{n/2}}p\frac{\ln (np)}{(\varepsilon^2np)}}\le e^{\eps^2p\binom{n}{\floor{n/2}}/400}.
\end{split}
\end{equation*}
Therefore, the probability $\Pi$ is at most $e^{-\eps^2p\binom{n}{\floor{n/2}}/400}=o(1)$, as required.

\medskip

To prove (ii), let $\cF\subseteq \cP(n,p)$ be an induced $K_{s,1,t}$-free family. As in the proof of Theorem \ref{s1tcount} (ii), let us consider the partition $\cF=\cD\cup \cU$ with $\cD=\{F\in \cF: \not\exists F_1,F_2,\dots, F_s$ an antichain with $F_i\subsetneq F \}$ and $\cU=\{F\in \cF: \not\exists F_1,F_2,\dots, F_t$ an antichain with $F_i\supsetneq F \}$. By part (i) of this theorem, both $\cD$ and $\cU$ are of size at most $(1+o(1))p\binom{n}{\lfloor n/2\rfloor}$ and thus $|\cF|\le  (2+o(1))p\binom{n}{\lfloor n/2\rfloor}$.
\end{proof}

\subsection{Results on trees - Theorem \ref{height2} and \ref{monotone}}

\begin{lemma}\label{tree}
Let $\oa{T}$ be a directed tree with $t$ edges that does not contain directed paths of length 2. Then there exists a $\delta>0$ such that the following holds: if $f(m)$ tends to infinity with $m$ and $\oa{G}_m$ is a directed graph on $m$ vertices with $f(m)\cdot m$ edges, then $\oa{G}$ contains $\delta f(m)^tm$ copies of $\oa{T}$. 
\end{lemma}

\begin{proof}
It is a well-known fact that every undirected graph contains a cut (i.e., a partition of its vertices into two and the edges between the parts) that contains at least half of its edges. This fact easily implies that the vertex set of $\oa{G}_m$ can be partitioned into $A\cup B$ such that there exist $\frac{f(m)m}{4}$ edges pointing from $A$ to $B$. Let $\oa{H}_0$ denote this directed subgraph of $\oa{G}_m$. For $i=1,2,\dots,t-1$ we remove the  set $U_i\subseteq V(\oa{H}_{i-1})$ of vertices that are adjacent to at most $\frac{f(m)}{8t}$ edges in $\oa{H}_{i-1}$. We also remove the set of adjacent edges to obtain $\oa{H}_i$. As for each $i$ we remove at most $\frac{f(m)m}{8t}$ edges, $\oa{H}_{t-1}$ contains at least $\frac{f(m)m}{8}$ edges. Also, by definition, if $\oa{uv}$ is an edge in $\oa{H}_i$, then both $u$ and $v$ are adjacent to at least $\frac{f(m)}{8t}$ edges in $\oa{H}_{i-1}$. Based on these properties, we can embed $\oa{T}$ to $\oa{G}_m$ greedily as follows: we fix an ordering $e_1,e_2,\dots, e_t$ of the edges of $\oa{T}$ such that $e_1,\dots,e_j$ is a tree for every $j=1,2,\dots,t$ and then embed $e_i$ to an edge of $\oa{H}_{t-i}$. There are at least $\frac{f(m)m}{8}$ choices for the image of $e_1$ and at least $\frac{f(m)}{8t}-i$ choices for the image of $e_i$ if $i\ge 2$. 
\end{proof}

\begin{thmn}[\ref{height2}]
Let $T$ be any height 2 tree poset of $t+1$ elements. Then for any $\varepsilon>0$ there exist $\delta>0$ and $n_0$ such that for any $n\ge n_0$ any family $\cF\subseteq 2^{[n]}$ of size $|\cF|\ge(1+\varepsilon)\binom{n}{\lfloor n/2\rfloor}$  contains at least $\delta n^t\binom{n}{\lfloor n/2\rfloor}$ copies of $T$.
\end{thmn}

\begin{proof}
For a family $\cF\subseteq 2^{[n]}$ of size at least $(1+\varepsilon)\binom{n}{\floor{n/2}}$, let us consider its directed comparability graph $\oa{G}_\cF$. This is the graph with vertex set $\cF$ where $\oa{FG}$ is an edge if and only if $F\subsetneq G$. By Theorem \ref{kchain}, $\oa{G}_\cF$ contains at least $\varepsilon\frac{n}{4}\binom{n}{\floor{n/2}}$ edges and thus Lemma \ref{tree} can be applied with $m=\binom{n}{\floor{n/2}}$ and $f(m)=\varepsilon\frac{n}{4}$. 
\end{proof}

\begin{proposition}\label{monx}
For any monotone tree poset $T$, we have
$x(T)=|T|-1+\sum(h(T)-r(\ell))$, where $h(T)$ denotes the height of $T$, the summation is over all leaves $\ell$ of $T$, and $r(\ell)$ is the rank of $\ell$, i.e., its distance from the root in the Hasse diagram of $T$ plus 1.
\end{proposition} 

\begin{proof}
Let us generate an embedding $i$ of $T$ into the $h(T)$ middle levels. Assume without loss of generality that $T$ is upward monotone, thus the root must be mapped to the lowest of the middle levels. We will refer to this level as the first middle level, and analogously the level above it is the second middle level, and the $k$th middle level is above the first level by $k-1$.

Observe that there are $(1+o(1))\binom{n}{\floor{n/2}}$ ways to pick the image of the root. Then we define $i$ by going through the elements of $T$ in a non-decreasing order with respect to the rank $r(p)$. That means that for every element $p$, when we decide where to embed it, we have already embedded its \emph{predecessor}, the (unique) element $p'$ that is smaller than $p$ such that there is no element $p''$ with $p'<p''<p$.

If the predecessor $p'$ of $p$ is mapped to the $j$th middle level and we want to map $p$ to the $k$th middle level, then the number of possibilities is $\Theta(n^{k-j})$. Let $f_i:T\rightarrow [h(T)]$ be a function defined by $f(p)=k$ if $i(p)$ is on the $k$th middle level. For a given $f:T\rightarrow [h(T)]$, the number of embeddings $i$ with $i=f_i$ is $\Theta(n^{\sum_{p\in T}(f(p)-f(p'))})$, where $p'$ is the predecessor of $p$ and the summation goes over all elements of $P$ apart from the root. Therefore, to obtain $x(T)$, we need to maximize $\Theta(n^{\sum_{p\in P}(f(p)-f(p'))})$. Clearly, when embedding a leaf $\ell$, we must have $f(\ell)=h(T)$. Finally, observe that if for some non-leaf $p$, we have $f(p)>f(p')+1$, then changing $f(p)$ to $f(p')+1$ cannot decrease the sum (and strictly increases it, if $p$ has at least two children). This shows that a function $f$ that maximizes the sum must satisfy $f(\ell)=h(T)$ for all leaves, and $f(p)=r(p)$ for all non-leaves.
\end{proof}

\begin{lemma}\label{embedding}
Assume that the families $2^{[n]}\supset \cF_1\supset \cF_2\supset \dots \supset \cF_h$ satisfy the following properties for some $\delta_1, \delta_2, \delta_3 >0$.

\begin{enumerate}[i)]
\item $|\cF_h|\ge \delta_1 \binom{n}{\lfloor{n/2\rfloor}}$.

\item For every $i=2,3,\dots,h$ and $F \in \cF_i$ there exist at least $\delta_2 n$ sets $F'\in \cF_{i-1}$ with $F\subsetneq F'$.

\item For every $i=2,3,\dots,h$ and $F\in \cF_i$ there exist at least $\delta_3 n^{i-1}$ sets $F'\in \cF_1$ with $F\subsetneq F'$.
\end{enumerate}

Then for any upward monotone tree poset $T$ of height $h(T)=h$, $\cF_1$ contains at least $\delta n^{x(T)}\binom{n}{\lfloor n/2\rfloor}$ copies of $T$, where $\delta>0$ is a constant depending on $\delta_1, \delta_2, \delta_3$ and $|T|$.
\end{lemma}

\begin{proof}
We generate embeddings of $T$ as  follows: we embed elements of $T$ according to their rank. The root of $T$ can be embedded to any set $F\in \cF_h$. By i), we have at least $\delta_1 \binom{n}{\lfloor{n/2\rfloor}}$ choices.

Any non-leaf element $x\in T$ with its predecessor embedded to $F_i \in \cF_i$ can be embedded to any set $F'\in \cF_{i-1}$ with $F_i \subsetneq F'$ that has not yet been used by the embedding. By ii), we have at least $\delta_2 n -|T|$ choices. If $x$ is of rank $r(x)$ then this process will embed it to $\cF_{h+1-r(x)}$.

If $\ell$ is a leaf vertex of rank $r(\ell)$ in $T$, then its predecessor is embedded into some $F\in\cF_{h+2-r(\ell)}$. We can embed $\ell$ to any superset of $F$ that has not yet been used. By iii), there exist at least $\delta_3 n^{h+1-r(\ell)}-|T|$ such sets.

Proposition \ref{monx} yields that the exponent of $n$ in the number of embeddings generated this way will be exactly $x(T)$. (We get a factor of $n$ for all vertices except for the root, and an additional $n^{h-r(\ell)}$ for leaves $\ell$.) A copy corresponds to at most $|T|!$ embeddings.
\end{proof}

\begin{thmn}[\ref{monotone}]
For any monotone tree poset $T$ and $\varepsilon>0$, there exist $\delta>0$ and $n_0$ such that for any $n\ge n_0$ any family $\cF\subseteq 2^{[n]}$ of size $|\cF|\ge (h(T)-1+\varepsilon)\binom{n}{\floor{n/2}}$ contains at least $\delta n^{x(T)}\binom{n}{\floor{n/2}}$ copies of $T$.
\end{thmn}

\begin{proof}
Throughout the proof we can assume that $n$ is sufficiently large. Let $h:=h(T)$. To prove the theorem, we will find families $\cF\supseteq \cF_1\supset \cF_2\supset \dots \supset \cF_h$ that satisfy the conditions of Lemma \ref{embedding} for sufficiently small $(\delta_1, \delta_2, \delta_3)$ depending on $\varepsilon$ and $T$. Let $\cF_1=\{G\in\cF:||G|-n/2|< n^{2/3}\}$.
Then $|\cF\setminus\cF_1| \le |\{G\subseteq [n]: ||G|-n/2|\ge n^{2/3}\}|\le o\left(\binom{n}{\floor{n/2}}\right)$.

For $i=2,3,\dots,h$ we define $\cF_i$ as follows. Let $\varepsilon':=\varepsilon/2h$. A set $F\in \cF_{i-1}$ is in $\cF_i$ if the expected number of sets of $\cF_{i-1}$ in a random chain between $F$ and $[n]$ is at least $1+\varepsilon'$. ($F$ counts for all chains.)

Let $\cG_i:=\cF_{i-1}\setminus \cF_i$. We partition the $n!$ full chains from $\emptyset$ to $[n]$ as follows. Let $\bC_G$ denote the collection of those chains whose smallest member from $\cG_i$ is $G$. Let $\bC_0$ denote the collection of those full chains that avoid $\cG_i$. If we take a random full chain from any of these collections, the expected value of members of $\cG_i$ in it is at most $1+\varepsilon'$. Therefore the expected value taken over all $n!$ full chains is also at most $1+\varepsilon'$. This means that $\lambda_n(\cG_i):=\sum_{F\in \cG_i}\frac{1}{\binom{n}{|F|}}\le 1+\varepsilon'$ and therefore $|\cG_i|\le (1+\varepsilon')\binom{n}{\floor{n/2}}$. These inequalities for $i=2,3,\dots,h$ together imply i) of Lemma~\ref{embedding}, if $\delta_1<\varepsilon/2$.

Consider a set $F\in \cF_{i-1}$. If $F$ has at most $\delta_2 n$ supersets of size between $|F|+1$ and $|F|+h-1$ in $\cF_{i-1}$ and at most $(\delta_2+\delta_3) n^h$ supersets in all of $\cF_{i-1}$, then the expected value of elements of $\cF_{i-1}$ in a chain from $F$ to $[n]$ is 
$$\sum_{F\subsetneq F'\in \cF_{i-1}} \frac{1}{\binom{n-|F|}{|F'\setminus F|}} \le\frac{\delta_2 n}{n-|F|}+\frac{(\delta_2+\delta_3) n^h}{\binom{n-|F|}{h}}\le 3\delta_2 + 3^h h! (\delta_2+\delta_3)$$
which is smaller than $\varepsilon'$ if $\delta_2$ and $\delta_3$ are a sufficiently small positive numbers compared to $h$ and $\varepsilon'$. This means that any $G\in \cF_i$ has either at least $\delta_2 n$ proper supersets of size at most $|G|+h-1$ in $\cF_{i-1}$ or at least $(\delta_2+\delta_3) n^h$ supersets in all of $\cF_{i-1}$.

The above statement trivially implies ii) of Lemma \ref{embedding}. We show that iii) of Lemma \ref{embedding} follows as well. Let $F\in \cF_i$ for some $2\le i \le h$. We need to find at least $\delta_3 n^{i-1}$ supersets of $F$ in $\cF_1$. If it has $(\delta_2+\delta_3) n^h$ supersets in $\cF_{i-1}$ then we are obviously done, so assume that this is not the case. We define a directed graph as follows. $F$ has at least $\delta_2 n$ supersets of size at most $|G|+h-1$ in $\cF_{i-1}$, let us direct an edge from $F$ to all of them. Each of these sets will have at least $\delta_2 n$ supersets of size at most $|G|+2(h-1)$ in $\cF_{i-2}$, direct an edge from the subset to the supersets. Continue this operation until we reach $\cF_1$. Now we have at least $(\delta_2 n)^{i-1}$ directed paths from $G$ to members of $\cF_1$. All these members are of size at most $|G|+(i-1)(h-1)$, therefore the number of paths leading to a single one of them is at most $((i-1)(h-1))^{i-1}<h^{2h}$. Thus we found at least $\frac{(\delta_2 n)^{i-1}}{h^{2h}}$ supersets of $G$ in $\cF_1$. This is more than $\delta_3 n^{i-1}$ if $\delta_{3}$ is sufficiently small compared to $\delta_2$ and $h$. 
\end{proof}

\subsection{Diamond-free families - Theorem \ref{diamond}}

Let us recall that $m_s= \lceil\log_2(s + 2)\rceil$ and $m^*_s=\min\{m: s\le  \binom{m}{\lceil m/2\rceil}\}$ and that for any integer $s\ge 2$, we have $x(D_s)=m_s$ and $x^*(D_s)=m^*_s$.

\begin{thmn}[\ref{diamond}]
\

(i) If $s \in [2^{m_s - 1} -1,2^{m_s} - \binom{m_s}{\lceil \frac{m_s}{2}\rceil}-1]$, then for any $\varepsilon>0$ there exists a $\delta>0$ such that every $\cF\subseteq 2^{[n]}$ with $|\cF|\ge (m_s+\varepsilon)\binom{n}{\lfloor n/2 \rfloor}$ contains at least $\delta \cdot n^{m_s-0.5}\binom{n}{\lfloor n/2 \rfloor}$ copies of $D_s$.

(ii) For any $\varepsilon>0$ there exists a $\delta>0$ such that every $\cF\subseteq 2^{[n]}$ with $|\cF|\ge (4+\varepsilon)\binom{n}{\lfloor n/2\rfloor}$ contains at least $\delta \cdot n^{3.5}\binom{n}{\lfloor n/2 \rfloor}$ induced copies of $D_4$. 

(iii) For any constant $c$ with $1/2<c<1$ there exists an integer $s_c$ such that if $s \ge s_c$ and $s\le c\binom{m^*_s}{\lfloor m^*_s/2\rfloor}$, then the following holds: for any $\varepsilon>0$  there exists a $\delta>0$ such that every $\cF\subseteq 2^{[n]}$ with $|\cF|\ge (m^*_s+\varepsilon)\binom{n}{\lfloor n/2\rfloor}$ contains at least $\delta \cdot n^{m^*_s-0.5}\binom{n}{\lfloor n/2 \rfloor}$ induced copies of $D_s$.
\end{thmn}

\begin{proof}
To prove (i), let us fix $\varepsilon>0$ and let $\cF\subseteq 2^{[n]}$ such that $|\cF|\ge (m_s+\varepsilon)\binom{n}{\lfloor n/2\rfloor}$. Let $\ell$ be the maximum integer such that $\binom{n}{\ell}\le \frac{\varepsilon}{8\sqrt{n}}\binom{n}{\lfloor n/2\rfloor}$. Then by Lemma \ref{binom}, we have $|\binom{[n]}{\le \ell}\cup \binom{[n]}{\ge n-\ell}|\le \frac{\varepsilon}{2}\binom{n}{\lfloor n/2\rfloor}$ and it is easy to verify that $\binom{n}{\ell}\ge \frac{\varepsilon}{10\sqrt{n}}\binom{n}{\lfloor n/2\rfloor}$ if $n$ is large enough. This implies that 
\begin{enumerate}
    \item 
    $\cF'=\cF \setminus (\binom{[n]}{\le \ell}\cup \binom{[n]}{\ge n-\ell})$ is of size at least $(m_s+\varepsilon/2)\binom{n}{\lfloor n/2\rfloor}$, in particular \\ $\sum_{F\in \cF}|F|!(n-|F|)!\ge (m_s+\varepsilon/2)n!$;
    \item
    for any $F\in \cF'$, we have $|F|!(n-|F)!\le \frac{10}{\varepsilon}\sqrt{n}\lfloor n/2\rfloor !\lceil n/2\rceil !$.
\end{enumerate}
We are going to count pairs $(F,\cC)$ with $F\in \cF'\cap \cC$ and $\cC$ is a maximal chain in $[n]$. On the one hand, using the first point above, this is clearly $\sum_{F\in \cF}|F|!(n-|F|)!\ge (m_s+\varepsilon/2)n!$. To count the pairs in another way, we will use the min-max partition of the maximal chains. Let $\bC_n$ denote the set of all $n!$ maximal chains in $[n]$, and let us partition $\bC_n$ into $\cup_{F,F'}\bC_{F,F'}$, where $F,F'$ run through all pairs $F\subseteq F'$ in $\cF'$ and $$\bC_{F,F'}:=\{\cC\in \bC_n: F ~\text{is minimal in} ~\cF'\cap \cC, F' ~\text{is maximal in} ~\cF'\cap \cC\}.$$
Those maximal chains that do not contain any $F\in \cF'$ are gathered in $\bC_\emptyset$. For any pair $F\subseteq F'$ in $\cF'$ let us write $b(F,F')=|\cF' \cap \{G: F\subsetneq G \subsetneq F'\}|$. Observe that the number of copies of $D_s$ in $\cF'$ is at least $\sum_{F,F'}\binom{b(F,F')}{s}$, in particular this number is at least $|\{(F,F'):b(F,F')\ge s\}|$. Finally, let $\bC_j=\bigcup_{b(F,F')=j}\bC_{F,F'}$ and $\bC_{<s}=\bC_\emptyset \cup\bigcup_{j=0}^{s-1}\bC_j$ and $\bC_{\ge s}=\bC_n\setminus \bC_{<s}$.

As $b(F,F')<s$ is equivalent to $[F,F']\cap \cF'$ being $D_s$-free, using Lemma \ref{nonindlem}, we obtain that the number of pairs $(F,\cC)$ with $F\in \cF'\cap \cC$ and $\cC\in \bC_{<s}$ is at most $m_sn!$. As a consequence, the number of pairs $(F,\cC)$ with $F\in \cF'\cap \cC$ and $\cC\in \bC_{\ge s}$ is at least $\frac{\varepsilon}{2}n!$. 
Observe that $|\bC_{F,F'}|\le |F|!(|F'|-|F|)!(n-|F')!$ and thus the number of pairs $(F'',\cC)$ with $F''\in \cF'\cap \cC$ and $\cC\in \bC_{F,F'}$ is at most $(|F'|-|F|+1)|F|!(|F'|-|F|)!(n-|F')!$. Also if $b(F,F')\ge s$, then $|F'|-|F|\ge m_s$. Fixing $|F|$ and using $\ell \le |F|,|F'|\le n-\ell$, it is easy to see that $(|F'|-|F|+1)|F|!(|F'|-|F|)!(n-|F'|)!$ is convex in $|F'|$ and thus it is maximized either at $|F'|-|F|=m_s$ or when $|F'|=n-\ell$. Plugging in one obtains that the maximum is taken when $|F'|-|F|=m_s$. Using 2. from above, we obtain that for one fixed $\bC_{F,F'} \subseteq \bC_{\ge s}$ the number of pairs $(F'',\cC)$ with $F''\in \cF'\cap \cC$ and $\cC\in \bC_{F,F'}$ is at most
\[
(m_s+1)m_s!|F|!(n-|F|-m_s)!\le (m_s+1)m_s!\frac{4^{m_s}}{n^{m_s}}|F|!(n-|F|)!\le (m_s+1)m_s!\frac{4^{m_s}10}{\varepsilon n^{m_s-0.5}}\lfloor n/2\rfloor!\lceil n/2\rceil!.
\]
Therefore the number of pairs $F,F'$ with $b(F,F')\ge s$ and thus the number of copies of $D_s$ in $\cF'$ is at least
\[
\frac{\varepsilon^2}{20(m_s+1)!4^{m_s}}n^{m_s-0.5}\binom{n}{\lfloor n/2\rfloor}
\]
as claimed. 

The proofs of (ii) and (iii) are basically the same. Instead of $b(F,F')$ one introduces $a(F,F')$, which is the maximum size of an antichain in $\cF'\cap [F,F']$, and partitions $\bC_n$ into $\bC_{<s}$ and $\bC_{\ge s}$ according to $a(F,F')$. If $a(F,F')<s$, then $\cF'\cap [F,F']$ is induced $D_s$-free, so for (ii), instead of Lemma \ref{nonindlem}, one applies Lemma \ref{plem} (i), and for (iii), one applies Lemma \ref{plem} (ii) to $\bC_{<s}$. The number of induced copies of $D_s$ is at least the number of pairs with $a(F,F')\ge s$ and the computation to obtain a lower bound for this number is the same as in (i).
\end{proof}

\textbf{Acknowledgement.} Research partially sponsored by the National Research, Development and Innovation Office -- NKFIH under the grants K 116769, K 132696, KH 130371, SNN 129364, FK 132060, and KKP-133819. Research of Vizer was supported by the J\'anos Bolyai Research Fellowship of the Hungarian Academy of Sciences and by the New National Excellence Program under the grant number \'UNKP-20-5-BME-45. Patk\'os acknowledges the financial support from the Ministry of Education and Science of the Russian Federation in the framework of MegaGrant no 075-15-2019-1926.

\bibliographystyle{abbrv}
\bibliography{bibo}
\end{document}